\documentclass[10pt,reqno]{amsart}
\usepackage{amsmath}
\usepackage{amssymb}
\usepackage{amsthm}

\newlength{\defbaselineskip}
\newcommand{\setlinespacing}[1]%
           {\setlength{\baselineskip}{#1 \defbaselineskip}}

\numberwithin{equation}{section}

\newtheorem{thm}{Theorem}[section]
\newtheorem{cor}[thm]{Corollary}
\newtheorem{lem}[thm]{Lemma}
\newtheorem{prop}[thm]{Proposition}

\theoremstyle{definition}

\theoremstyle{remark}
\newtheorem{rem}[thm]{Remark}
\numberwithin{equation}{section}

\begin{document}

\title[Strichartz estimates in Wiener amalgam spaces]
{Strichartz estimates for the Schr\"odinger propagator in Wiener amalgam spaces}

\author{Seongyeon Kim, Youngwoo Koh and Ihyeok Seo}

\thanks{Y. Koh was supported by NRF Grant 2016R1D1A1B03932049 (Republic of Korea).
I. Seo was supported by the NRF grant funded by the Korea government(MSIP) (No. 2017R1C1B5017496).}

\subjclass[2010]{Primary: 35B45, 35Q40; Secondary: 42B35 }
\keywords{Strichartz estimates, Schr\"odinger propagator, Wiener amalgam spaces}

\address{Department of Mathematics, Sungkyunkwan University, Suwon 16419, Republic of Korea}
\email{synkim@skku.edu}

\address{Department of Mathematics Education, Kongju National University, Kongju 32588,
Republic of Korea}
\email{ywkoh@kongju.ac.kr}

\address{Department of Mathematics, Sungkyunkwan University, Suwon 16419, Republic of Korea}
\email{ihseo@skku.edu}

\begin{abstract}
In this paper we study the Strichartz estimates for the Schr\"odinger propagator
in the context of Wiener amalgam spaces
which, unlike the Lebesgue spaces, control the local regularity of a function and its decay at infinity separately.
This separability makes it possible to perform a finer analysis of the local and global behavior of the propagator.
Our results improve some of the classical ones in the case of large time.
\end{abstract}

\maketitle

\section{Introduction}

Consider the following Cauchy problem for the Schr\"odinger equation
\begin{equation} \label{eq}
\begin{cases}
i\partial_t u + \Delta u = 0,\\
u(x,0) = f(x),
\end{cases}
\end{equation}
with $(x,t)\in\mathbb{R}^n\times\mathbb{R}$, $n \ge 1$.
Applying the Fourier transform to \eqref{eq}, the solution $u(x,t)$ is given by
\begin{equation}\label{sol}
e^{it\Delta}f(x)= \frac{1}{(2\pi)^n} \int_{\mathbb{R}^n} e^{i (x\cdot \xi - t|\xi|^2)} \hat f (\xi)d\xi.
\end{equation}
Here the Fourier multiplier $e^{it\Delta}$ is called the Schr\"odinger propagator.

The following space-time integrability of \eqref{sol} in $L^p$ spaces has been intensively studied in the last forty years:
\begin{equation} \label{classi}
\|e^{it\Delta} f\|_{L_t^q L_x^r} \lesssim \|f\|_{L^2}
\end{equation}
for $(q,r)$ Schr\"odinger admissible, i.e., for
\begin{equation}\label{admiss}
q,r\geq 2,\quad \frac{2}{q} + \frac{n}{r} = \frac{n}{2},\quad (q,r,n)\neq(2,\infty,2).
\end{equation}
See \cite{Str,GV,M-S,KT} and references therein.

In this paper we consider these space-time estimates, known as \textit{Strichartz estimates},
in Wiener amalgam spaces which, unlike the $L^p$ spaces, control the local regularity of a function and its decay at infinity separately.
This separability makes it possible to perform a finer analysis of the local and global behavior of the solution.
These aspects were originally pointed out in the papers \cite{CN,CN2,CN3}. 
These spaces were first introduced by Feichtinger \cite{F}
and have already appeared as a technical tool in the study of partial differential equations (\cite{T}).

To begin with, let us recall the definition of Wiener amalgam spaces.
Let $\varphi \in C_0^{\infty}$ be a test function satisfying $\|\varphi\|_{L^2} =1$.
Let $1\le p,q \le \infty$. 
Then the Wiener amalgam space $W(L^p, L^q)$ 
is defined as the space of functions $f \in L_{\textrm{loc}}^p$ 
equipped with the norm
$$\|f\|_{W(L^p, L^q)}= \big\|\, \|f \tau_x \varphi\|_{L^p} \big\|_{L_x^q},$$
where $\tau_x \varphi(\cdot) = \varphi(\cdot-x)$.
Here different choices of $\varphi$ generate the same space and yield equivalent norms.
The Wiener amalgam space can be also seen as a natural extension of $L^p$ space in view of $W(L^p, L^p)=L^p$.
More generally, the Wiener amalgam space $W(A,B)$ for Banach spaces $A$ and $B$ is defined in the same way.

In \cite{CN3}, Cordero and Nicola established the following estimates for Schr\"odinger admissible $(q, r)$:
\begin{equation} \label{C-N}
\|e^{it\Delta}f\|_{W(L^{\infty}, L^{q})_t W(L^{2}, L^{r})_x} \lesssim \|f\|_{L^2}.
\end{equation}
By complex interpolation (see \eqref{inter}) between \eqref{C-N} and \eqref{classi}, they obtained further estimates
\begin{equation} \label{C-N2}
\|e^{it\Delta}f\|_{W(L^{\widetilde{q}}, L^{q})_t W(L^{\widetilde{r}}, L^{r})_x} \lesssim \|f\|_{L^2}
\end{equation}
for $(\tilde q, \tilde r)$ and $(q,r)$ satisfying
$1\leq \tilde q, \tilde r \leq \infty$, $2\leq q, r \leq \infty$,
$$\tilde r \leq r,\quad
\frac{2}{q}+\frac{n}{r}\leq\frac{n}{2}\leq\frac{2}{\tilde q} + \frac{n}{\tilde r},$$
$\widetilde{r},r<\infty$ if $n=2$, and if $n\geq3$, $\widetilde{r}\leq2n/(n-2)$.
As mentioned in \cite{CN3}, these estimates say that the analysis of the local regularity of the Schr\"odinger propagator
is quite independent of its decay at infinity since there are no relations between the pairs $(\tilde q, \tilde r)$ and $(q,r)$
other than $\tilde r \leq r$.
(See also \cite{CN,S} for related results.)

Our goal in this paper is to provide a picture of the Strichartz estimates in Wiener amalgam spaces for the Schr\"odinger
propagator on initial data with regularity.
We attempt to obtain
\begin{equation}\label{TT}
\| e^{it\Delta} f\|_{W(L^{\tilde q} ,L^q)_t W(L^{\tilde r} ,L^r)_x} \lesssim \|f\|_{\dot H^{\sigma}}
\end{equation}
with the homogeneous Sobolev norm $\|f\|_{\dot H^{\sigma}} = \||\nabla|^{\sigma}f\|_{L^2}$, $\sigma>0$.
In the case of the Lebesgue space estimates,
\begin{equation}\label{efv}
\|e^{it\Delta}f\|_{L_t^q L_x^r} \lesssim \|f\|_{\dot H^{\sigma}},
\end{equation}
where $0<\sigma<n/2$, $q\geq2$ and
\begin{equation}\label{cc1}
\frac2q+\frac nr=\frac n 2-\sigma,
\end{equation}
one can proceed by first considering initial data which are frequency localized to annuli
and then use Littlewood-Paley theory to obtain the desired estimates for general data.
It seems difficult to proceed in this way in the case
of the generalized estimates \eqref{TT}.
Here we bypass Littlewood-Paley theory to obtain
directly the estimates \eqref{TT}. The key ingredient in our approach is the availability
of estimates for the integral kernel of the Fourier multiplier $e^{it\Delta}|\nabla|^{-\sigma}$.
Our main result is the following theorem.

\begin{thm}\label{thm}
Let $n\geq1$.
Let  $2\le\tilde q<q<\infty$, $2\le \tilde r, r \le \infty$ and $\max\{0,(n-2)/4\}<\sigma<n/2$.
Assume that $(\tilde q , \tilde r)$ and $(q,r)$ satisfy
\begin{equation} \label{c1}
\frac{2}{\tilde q} + \frac{n-1}{\tilde r} > \frac{n}{2} -\sigma
\end{equation}
and
\begin{equation} \label{c2}
	\frac{2}{q} + \frac{n}{r} = \frac{n}{2}-\sigma-\frac{n-1}{\tilde r}.
\end{equation}
Then we have
\begin{equation}\label{T}
\| e^{it\Delta} f\|_{W(L^{\tilde q} ,L^q)_t W(L^{\tilde r} ,L^r)_x} \lesssim \|f\|_{\dot H^{\sigma}}.
\end{equation}
\end{thm}

\begin{rem}
In particular, when $\tilde r =\infty$, the condition \eqref{c2} becomes the condition \eqref{cc1}
and this case is then comparable to the classical estimate \eqref{efv}.
Roughly speaking, the estimate \eqref{T} in this case shows that the $W(L^{\infty} ,L^r)_x$-norm of the solution has a $L_t^q$-decay at infinity.
Hence, our result is better than the classical one for large time
since the classical $L_x^r$ norm in \eqref{efv} is rougher than the $W(L^{\infty} ,L^r)_x$-norm,
although locally the classical $L_t^q$ regularity is replaced by
$L_t^{\widetilde{q}}$ with $\widetilde{q}<q$.
\end{rem}

\begin{rem} \label{rmk}
From complex interpolation (see \eqref{inter}) between bilinear form estimates given from \eqref{T} and \eqref{C-N2},
we can obtain further estimates.
See Section \ref{sec4} for details.
In a different way, one can also easily obtain further estimates by the interpolation between \eqref{T} and \eqref{efv} with the same $\sigma$.
We omit the details.
Finally, we can trivially increase $q,r$ and diminish $\tilde q, \tilde r$ in \eqref{T} by using the inclusion relation (see \eqref{inclusion})
of Wiener amalgam spaces.

\end{rem}

The outline of this paper is as follows:
In Section \ref{sec2} we prove Theorem \ref{thm} assuming Proposition \ref{prop}
which shows fixed-time estimates for the integral kernel of the Fourier multiplier $e^{it\Delta}|\nabla|^{-\sigma}$.
Proposition \ref{prop} is proved in Section \ref{sec3}.
We consider Remark \ref{rmk} in Section \ref{sec4}.

Throughout this paper, the letter $C$ stands for a positive constant which may be different
at each occurrence.
We also denote $A\lesssim B$ to mean $A\leq CB$ with unspecified constants $C>0$.

\section{Proof of Theorem \ref{thm}} \label{sec2}
In this section we prove Theorem \ref{thm}.
First we list some basic properties of Wiener amalgam spaces which will be frequently used in the sequel.
We refer to \cite{F,F2,F3,H} for details:

\begin{lem}
Let $1\le p_i ,q_i \le \infty$ for $i=0,1,2$. Then the followings hold:
\begin{itemize}
\item Inclusion: if\, $p_1 \ge p_2$ and $q_1 \le q_2$,
\begin{equation} \label{inclusion}
W(L^{p_1}, L^{q_1}) \subset W(L^{p_2}, L^{q_2}).
\end{equation}
\item Convolution: if\, $1/p_2 + 1 = 1/p_0 + 1/p_1$ and $1/q_2 + 1 = 1/q_0 + 1/q_1$,
\begin{equation}\label{y-ineq}
W(L^{p_0}, L^{q_0}) * W(L^{p_1} , L^{q_1}) \subset W(L^{p_2} , L^{q_2}).
\end{equation}
More generally, if\, $B_0 * B_1 \subset B_2$ and $C_0 * C_1 \subset C_2$,
\begin{equation}\label{con}
W(B_0 , C_0) * W(B_1 , C_1) \subset W(B_2, C_2)
\end{equation}
for Banach spaces $B_i$ and $C_i$, $i=0,1,2$.

\

\item Duality: if\, $1\le p, q<\infty$,
\begin{equation*}
W(L^{p}, L^{q})' = W(L^{p'}, L^{q'}).
\end{equation*}
Here, $p', q'$ are conjugate exponents.

\

\item Complex interpolation: if\, $q_0 < \infty$ or $q_1 < \infty$,
\begin{equation} \label{inter}
[W(L^{p_0}, L^{q_0}), W(L^{p_1}, L^{q_1})]_{[\theta]} = W(L^p, L^q)
\end{equation}
whenever
\begin{equation} \label{theta}
\frac{1}{p} = \frac{\theta}{p_0} + \frac{1-\theta}{p_1},  \quad \frac{1}{q} = \frac{\theta}{q_0} + \frac{1-\theta}{q_1}, \quad 0<\theta<1.
\end{equation}
\end{itemize}
\end{lem}

\subsection{Proof of Theorem \ref{thm}}\label{subsec2.1}
Now we turn to the proof of Theorem \ref{thm}.
To prove \eqref{T}, we can apply the standard TT* argument because of the H\"older's type inequality
\begin{equation} \label{Holder}
|\langle F,G \rangle_{L^2_{x,t}}|
\leq \|F\|_{W(L^{\tilde q}, L^{q})_t W(L^{\tilde r}, L^r)_x} \|G\|_{W(L^{{\tilde q}'} , L^{q'})_t W(L^{{\tilde r}'}, L^{r'})_x}
\end{equation}
which can be proved directly from the definition of these spaces.
Indeed, using this inequality, we first see that \eqref{T} is equivalent to
\begin{equation}\label{T*}
\bigg\| \int_{\mathbb{R}} e^{-is\Delta} |\nabla|^{-\sigma} F(\cdot,s) ds \bigg\|_{L^2_x} \lesssim \| F\|_{W(L^{{\tilde q}'} , L^{q'})_t W(L^{{\tilde r}'}, L^{r'})_x}.
\end{equation}
To show this, note first that
\begin{align} \label{dualT*}
\nonumber
\bigg\| \int_{\mathbb{R}} e^{-is\Delta} |\nabla|^{-\sigma} F(\cdot,s) ds \bigg\|_{L^2_x}
&=\sup_{\|f\|_{L^2}=1}\bigg|\bigg\langle \int_{\mathbb{R}} e^{-is\Delta} |\nabla|^{-\sigma} F(\cdot,s) ds, \,\,f \bigg\rangle_{L^2_x}\bigg|\\
&=\sup_{\|f\|_{L^2}=1} \big| \big\langle F, \,\, e^{is\Delta} |\nabla|^{-\sigma} f \big\rangle_{L^2_{y,s}} \big|.
\end{align}
Hence, if \eqref{T} holds then the right-hand side of \eqref{dualT*} is no larger than
$$\| F\|_{W(L^{{\tilde q}'},L^{q'})_t W(L^{{\tilde r}'},L^{r'})_x}$$
and \eqref{T*} follows.
Conversely, if \eqref{T*} holds then the entire expression in the right-hand side of \eqref{dualT*} is no larger than $C\|F\|_{W(L^{{\tilde q}'} , L^{q'})_t W(L^{{\tilde r}'}, L^{r'})_x}$, which implies \eqref{T}. Thus \eqref{T} and \eqref{T*} are equivalent.
Clearly, applying first \eqref{T} and then \eqref{T*} yields
\begin{equation}\label{TT*}
\bigg\| \int_\mathbb{R} e^{i(t-s)\Delta} |\nabla|^{-2{\sigma}} F(\cdot,s)  ds \bigg\|_{W(L^{\tilde q} ,L^q)_t W(L^{\tilde r} ,L^r)_x} \lesssim \|F\|_{W(L^{\tilde q'} ,L^{q'})_t W(L^{\tilde r'} ,L^{r'})_x}.
\end{equation}
Finally, if \eqref{TT*} holds then
\begin{align*}
\bigg\| \int_{\mathbb{R}}  e^{-is\Delta}  |\nabla|^{-\sigma} F(\cdot,s) ds \bigg\|^2_{L^2_x}
& = \bigg\langle \int_{\mathbb{R}} e^{-it\Delta} |\nabla|^{-\sigma} F(\cdot,t) dt, \int_{\mathbb{R}} e^{-is\Delta} |\nabla|^{-\sigma} F(\cdot,s) ds \bigg\rangle_{L^2_x} \\
& =\bigg\langle F, \int_\mathbb{R} e^{i(t-s)\Delta} |\nabla|^{-2{\sigma}} F(\cdot,s)  ds \bigg\rangle_{L^2_{x,t}} \\
&\lesssim\|F\|^2_{W(L^{{\tilde q}'} , L^{q'})_t W(L^{{\tilde r}'}, L^{r'})_x}
\end{align*}
using the H\"older inequality \eqref{Holder}, which is \eqref{T*}.
Consequently, the estimates \eqref{T}, \eqref{T*} and \eqref{TT*} are equivalent to each other.

From now on, we shall prove \eqref{TT*}.
We first write the integral kernel $K_t(x)$ of the multiplier $e^{it\Delta} |\nabla|^{-2{\sigma}}$ as
\begin{equation} \label{ker}
K_t (x) := \frac{1}{(2\pi)^n}\int_{\mathbb{R}^n} e^{i(x \cdot \xi - t|\xi|^2)} \frac{d\xi}{|\xi|^{2\sigma}}.
\end{equation}
Then \eqref{TT*} is rephrased as follows:
\begin{equation}\label{conv}
\bigg\| \int_{\mathbb{R}} (K_{t-s} * F(\cdot,s))(x) ds \bigg\|_{W(L^{\tilde q} ,L^q)_t W(L^{\tilde r},L^r)_x} \lesssim \|F\|_{W(L^{\tilde q'} ,L^{q'})_t W(L^{\tilde r'},L^{r'})_x}.
\end{equation}
From now on, we will obtain \eqref{conv}.
By Minkowski's inequality and the convolution relation \eqref{y-ineq}, it follows that
\begin{align}\label{bfHLS}
\nonumber \bigg\| \int_{\mathbb{R}} (K_{t-s} * &F(\cdot,s))(x) ds \bigg\|_{W(L^{\tilde q} ,L^q)_t W(L^{\tilde r} ,L^r)_x} \\
\nonumber
&\le \bigg\| \int_{\mathbb{R}} \|K_{t-s} * F(\cdot,s)\|_{W(L^{\tilde r} ,L^r)_x} ds \bigg\|_{W(L^{\tilde q},L^q)_t} \\
&\le \bigg\| \int_{\mathbb{R}} \|K_{t-s}\|_{W(L^{\frac{\tilde r}{2}} ,L^{\frac{r}{2}})_x} \|F(\cdot,s)\|_{W(L^{\tilde r'} ,L^{r'})_x} ds\bigg\|_{W(L^{\tilde q} ,L^q)_t}.
\end{align}
Recall the Hardy-Littlewood-Sobolev fractional integration theorem (see e.g. \cite{St}, p. 119) in dimension $1$:
\begin{equation}\label{hls}
L^p(\mathbb{R})\ast L^{1/\alpha,\infty}(\mathbb{R})\hookrightarrow L^q(\mathbb{R})
\end{equation}
for $0<\alpha < 1$ and $1\le p<q<\infty$ with $\frac{1}{q} +1 = \frac{1}{p} + \alpha$.
Applying \eqref{hls} with $p=q'$ and $\alpha = \frac{2}{q}$, and Young's inequality,
the convolution relations \eqref{con} then give
\begin{equation*}
W(L^{\frac{\tilde q}{2}}, L^{\frac{q}{2}, \infty})_t * W(L^{\tilde q'}, L^{q'})_t \subset W(L^{\tilde q}, L^q )_t
\end{equation*}
for
\begin{equation}\label{you}
2\leq\widetilde{q}\leq\infty\quad \text{and}\quad 2<q<\infty.
\end{equation}
Hence we get
\begin{align}\label{bfHLS2}
\nonumber\bigg\| \int_{\mathbb{R}} \|K_{t-s}\|_{W(L^{\frac{\tilde r}{2}} ,L^{\frac{r}{2}})_x} &\|F(\cdot,s)\|_{W(L^{\tilde r'} ,L^{r'})_x} ds\bigg\|_{W(L^{\tilde q} ,L^q)_t}\\
\lesssim&
\| K_t \|_{{W(L^{\frac{\tilde q}{2}}, L^{\frac{q}{2},\infty})_t}{W(L^{\frac{\tilde r}{2}} ,L^{\frac{r}{2}})_x}}
\|F\|_{W(L^{\tilde q'} ,L^{q'})_t W(L^{\tilde r'},L^{r'})_x}.
\end{align}
Combining \eqref{bfHLS} and \eqref{bfHLS2}, we now obtain the desired estimate \eqref{conv} if
\begin{equation}\label{t}
\| K_t \|_{{W(L^{\frac{\tilde q}{2}}, L^{\frac{q}{2},\infty})_t}{W(L^{\frac{\tilde r}{2}} ,L^{\frac{r}{2}})_x}} < \infty
\end{equation}
for $(\tilde q , \tilde r)$ and $(q,r)$ satisfying the same conditions as in Theorem \ref{thm}.
To show \eqref{t}, we use the following fixed-time estimates for the integral kernel which will be proved in Section \ref{sec3}:

\begin{prop} \label{prop}
Let $n\geq1$.
Let $2\leq \tilde r,r \leq \infty$ and $0<\sigma<n/2$.
Assume that
\begin{equation} \label{c3}
\frac{n-1}{\tilde r} + \frac{n}{r} <\sigma \quad\text{if}\quad 0 < \sigma \le n/4,
\end{equation}
and
\begin{equation} \label{c4}
\frac{n-1}{\tilde r} + \frac{n}{r} < \frac{n}{2} - \sigma\quad\text{if}\quad n/4 \le \sigma < n/2.
\end{equation}
Then we have
	\begin{equation}\label{x}
	\|K_t\|_{W(L^{\frac{\tilde r}{2}} ,L^{\frac{r}{2}})_x} \lesssim
	\begin{cases}
	|t|^{-\frac{n}{2} + \sigma+\frac{n-1}{\tilde r}} \quad\text{if} \quad 0<|t| \le 1, \\
	|t|^{-\frac{n}{2} + \sigma+ \frac{n-1}{\tilde r}+\frac{n}{r}} \quad \text{if} \quad |t| \ge 1.
	\end{cases}
	\end{equation}	
\end{prop}

To begin with, we set $h(t)= \|K_t\|_{W(L^{\frac{\tilde r}{2}} ,L^{\frac{r}{2}})_x}$
and choose $\varphi(t) \in C_0^{\infty}(\mathbb{R})$ supported on $\{t\in\mathbb{R}:|t|\le1\}$.
To calculate $\| h \|_{{W(L^{\frac{\tilde q}{2}}, L^{\frac{q}{2},\infty})_t}}$ using \eqref{x},
we divide $\|h \tau_k \varphi \|_{L_t^{\tilde q/2}}$ into three cases, $|k|\le 1$, $1\le|k|\le 2$ and $|k|\ge 2$.

First we consider the case $|k|\le 1$. By using \eqref{x} and the support condition of $\varphi$,
\begin{equation}\label{dfg}
\|h \tau_k \varphi \|_{L_t^{\tilde q/2}}^{\tilde q/2} \lesssim \int_{0<|t|\le1} |t|^{\frac{\tilde q}{2} (-\frac{n}{2} + \sigma+\frac{n-1}{\tilde r})} dt + \int_{1\le|t|\le|k|+1} |t|^{\frac{\tilde q} {2} (-\frac{n}{2} + \sigma+ \frac{n-1}{\tilde r}+\frac{n}{r})} dt.
\end{equation}
Since $\frac{\tilde q}{2} (-\frac{n}{2} + \sigma+\frac{n-1}{\tilde r})+1>0$ by the condition \eqref{c1},
the first integral in the right-hand side of \eqref{dfg} is trivially finite.
The second inequality is bounded as follows:
\begin{align}\label{ert}
\nonumber\int_{1\le|t|\le|k|+1} |t|^{\frac{\tilde q} {2} (-\frac{n}{2} + \sigma+ \frac{n-1}{\tilde r}+\frac{n}{r})} dt
&\lesssim\frac{(|k|+1)^{\frac{\tilde q} {2} (-\frac{n}{2} + \sigma+ \frac{n-1}{\tilde r}+\frac{n}{r})+1} - 1 }
{\frac{\tilde q} {2} (-\frac{n}{2} + \sigma+ \frac{n-1}{\tilde r}+\frac{n}{r})+1}\\
&\lesssim |k|.
\end{align}
Indeed, since ${\frac{\tilde q}{2}(- \frac{n}{2} + \sigma+ \frac{n-1}{\tilde r}+\frac{n}{r})}<0$
by the condition \eqref{c2}, the second inequality in \eqref{ert} follows easily from the mean value theorem.
Hence we get
\begin{equation*}
\|h \tau_k \varphi \|_{L_t^{\tilde q/2}}^{\tilde q/2} \lesssim 1
\end{equation*}
when $|k|\le 1$.
The other cases $1\le|k|\le 2$ and $|k|\ge 2$ are handled in the same way:

\begin{align}
\nonumber
\|h \tau_k \varphi \|_{L_t^{\tilde q/2}}^{\tilde q/2}&\lesssim \int_{|k|-1\le|t|\le1} |t|^{\frac{\tilde q}{2} (- \frac{n}{2} + \sigma+\frac{n-1}{\tilde r})} dt + \int_{1\le|t|\le|k|+1} |t|^{\frac{\tilde q}{2} (-\frac{n}{2} + \sigma+ \frac{n-1}{\tilde r}+\frac{n}{r})} dt \\
\nonumber
&\lesssim 1+ |k| \\
\nonumber
&\lesssim 1
\end{align}
when $1\le|k|\le 2$, and when $|k|\ge 2$

\begin{align}
\nonumber
\|h \tau_k \varphi \|_{L_t^{\tilde q/2}}^{\tilde q/2}&\lesssim \int_{{|k|-1}\le|t|\le{|k|+1}} |t|^{\frac{\tilde q} {2} (-\frac{n}{2} + \sigma+ \frac{n-1}{\tilde r}+\frac{n}{r})} dt \\
\nonumber
\label{t-3}
&\lesssim\frac{(|k|+1)^{\frac{\tilde q}{2} (- \frac{n}{2} + \sigma+ \frac{n-1}{\tilde r}+\frac{n}{r})+1} - (|k|-1)^{\frac{\tilde q}{2} (- \frac{n}{2} + \sigma+ \frac{n-1}{\tilde r}+\frac{n}{r}) +1}}{\frac{\tilde q} {2} (-\frac{n}{2} + \sigma+ \frac{n-1}{\tilde r}+\frac{n}{r})+1} \\
\nonumber
&\lesssim (|k|-1)^{\frac{\tilde q}{2} (-\frac{n}{2} + \sigma+ \frac{n-1}{\tilde r}+\frac{n}{r})}.
\end{align}
Consequently, we get

\begin{equation} \label{local_t}
\|h\tau_k\varphi\|_{L_t^{\tilde q/2}} \lesssim
\begin{cases}
1 \quad\text{if}\quad |k| \leq 2,\\
(|k|-1)^{- \frac{n}{2} + \sigma+ \frac{n-1}{\tilde r}+\frac{n}{r}} \quad\text{if}\quad |k| \geq 2.
\end{cases}
\end{equation}
By \eqref{local_t}, $\|h\tau_k\varphi\|_{L_t^{\tilde q/2}}$  belongs to $L^{\frac{q}{2}, \infty}_k$
since we are assuming the condition \eqref{c2} which is equivalent to $\frac{2}{q} = \frac{n}{2}-\sigma-\frac{n-1}{\tilde r} -\frac{n}{r}$.
This implies
$$\| h \|_{{W(L^{\frac{\tilde q}{2}}, L^{\frac{q}{2},\infty})_t}}<\infty$$
for $(\tilde q , \tilde r)$ and $(q,r)$ satisfying the conditions \eqref{you}, $2\leq \tilde r,r \leq \infty$, \eqref{c1},
\eqref{c2}, \eqref{c3}, \eqref{c4}.

Combining \eqref{c2} and \eqref{c3}, we see $2/q>n/2-2\sigma$.
Since $2<q<\infty$ by \eqref{you}, this implies the restriction $\sigma>(n-2)/4$.
On the other hand, the conditions \eqref{c2} and \eqref{c4} imply $2/q>0$.
There is no restriction in this case.
Finally, combining \eqref{c1} and \eqref{c2}, we see $2/\widetilde{q}>2/q+n/r$. Hence $q>\widetilde{q}$.
Therefore, we get the desired estimate \eqref{t} for $(\tilde q , \tilde r)$ and $(q,r)$ satisfying
$2\le\tilde q<q<\infty$, $2\le \tilde r, r \le \infty$, \eqref{c1}, \eqref{c2}
when $(n-2)/4<\sigma<n/2$.
This completes the proof of Theorem \ref{thm}.


\section{Proof of Proposition \ref{prop}} \label{sec3}
In this section, we prove Proposition \ref{prop} by making use of the following lemma.
(As mentioned in \cite{BBCRV}, this lemma is seen to be sharp in the case $\gamma=n/2$.)

\begin{lem} (\cite{BBCRV}, Lemma 2.2)
Let $n\geq1$ and $0< \gamma < n$. Then if $t \neq 0$
\begin{equation}\label{kerbd}
\left| \int_{\mathbb{R}^n} e^{i(x \cdot \xi - t|\xi|^2)} \frac{d\xi}{|\xi|^{\gamma}} \right| \lesssim
\begin{cases}
\frac{|t|^{-(n/2-\gamma)}}{(|x|^2 + |t|)^{\gamma/2}} \quad\text{if} \quad 0 < \gamma \le \frac{n}{2},\\
\frac{1}{(|x|^2 + |t|)^{(n-\gamma)/2}} \quad \text{if} \quad \frac{n}{2} \le \gamma < n.
\end{cases}
\end{equation}
\end{lem}

We prove \eqref{x} only for the case $\widetilde{r},r<\infty$
because the other cases $\widetilde{r}=\infty$ or $r=\infty$ follow clearly and more easily from the same argument.
We divide cases into $0 < \sigma \le n/4$ and $n/4 \le \sigma < n/2$.

\subsection{The case $0 < \sigma \le n/4$}
From \eqref{ker} and \eqref{kerbd} with $\gamma=2\sigma$, we see
\begin{equation} \label{upbd}
|K_t (x)| \lesssim
\begin{cases}
|t|^{-(\frac{n}{2} - 2\sigma)}|t|^{-\sigma} \quad \text{if} \quad |x| \le \sqrt{t},\\
|t|^{-(\frac{n}{2} - 2\sigma)}|x|^{-2 \sigma} \quad \text{if} \quad |x| \ge \sqrt{t}.
\end{cases}
\end{equation}
To calculate $\|K_t\|_{W(L^{\frac{\tilde r}{2}}, L^{\frac{r}{2}})_x}$ using \eqref{upbd},
we choose $\varphi(x) \in C_0^{\infty}(\mathbb{R}^n)$ supported on $\{x\in\mathbb{R}^n:|x|\le1\}$
and divide $\|K_t \tau_y \varphi\|_{L_x^{\tilde r/2}}$ into two cases, $|y|\le 1+\sqrt{t}$ and $|y|\ge 1+\sqrt{t}$.

First we consider the case $|y|\le 1+\sqrt{t}$.
By using \eqref{upbd} and the support condition of $\varphi$,
\begin{align}
\nonumber
|t|^{\frac{\widetilde{r}}{2}(\frac{n}{2} - 2\sigma)}\|K_t \tau_y \varphi\|_{L_x^{\tilde r/2}}^{\tilde r/2}
 &\lesssim \int_{|y|-1 \le |x| \le \sqrt{t}} |t|^{-\frac{\sigma\tilde r}{2}} dx + \int_{\sqrt{t} \le |x| \le |y|+1} |x|^{-\sigma\tilde r} dx \\
\nonumber
&= |t|^{-\frac{\sigma\tilde r}{2}} \int_{|y|-1}^{\sqrt{t}} \rho^{n-1} d\rho + \int_{\sqrt{t}}^{|y|+1} \rho^{-\sigma\tilde r+n-1} d\rho \\
\nonumber
&=|t|^{-\frac{\sigma\tilde r}{2}} \frac{ {\sqrt{t}}^n - (|y|-1)^n}{n} +
\frac { (|y|+1)^{-\sigma \tilde r + n} - \sqrt{t}^{-\sigma\tilde r + n} }{-\sigma \tilde r +n}.
\end{align}
Since $n-1\ge 0$ and $-\sigma \tilde r + n-1 <0$ from \eqref{c3},
by applying the mean value theorem as before, we now see
\begin{align}\label{I}
\nonumber
|t|^{\frac{\widetilde{r}}{2}(\frac{n}{2} - 2\sigma)}\|K_t \tau_y \varphi\|_{L_x^{\tilde r/2}}^{\tilde r/2}
&\lesssim |t|^{\frac{\tilde r}{2}(-\sigma+\frac{n-1}{\tilde r})} (\sqrt{t}-|y|+1) + |t|^{\frac{\tilde r}{2}(-\sigma+\frac{n-1}{\tilde r})} (|y|+1-\sqrt{t}) \\
&= 2|t|^{\frac{\tilde r}{2}(-\sigma+\frac{n-1}{\tilde r})}
\end{align}
when $|y|\le 1+\sqrt{t}$.
The other case $|y|\ge 1+\sqrt{t}$ is handled in the same way:
\begin{align}\label{II}
\nonumber
|t|^{\frac{\widetilde{r}}{2}(\frac{n}{2} - 2\sigma)}\|K_t \tau_y \varphi\|_{L_x^{\tilde r/2}}^{\tilde r/2}
&\lesssim \int_{|y|-1 \le |x| \le |y|+1} |x|^{-\sigma\tilde r} dx \\
\nonumber
&=\int_{|y|-1}^{|y|+1} \rho^{-\sigma\tilde r+n-1} d\rho \\
\nonumber
&=\frac { (|y|+1)^{-\sigma \tilde r +n} - (|y|-1)^{-\sigma \tilde r + n} }{-\sigma \tilde r + n} \\
&\lesssim (|y|-1)^{{\tilde r} (-\sigma+\frac{n -1}{\tilde r})}.
\end{align}

By \eqref{I} and \eqref{II}, it follows now that
\begin{align}\label{qwd}
\nonumber\|K_t\|_{W(L^{\frac{\tilde r}{2}} ,L^{\frac{r}{2}})_x}^{r/2}
&= \big\|\, \|K_t \tau_y \varphi\|_{L_x^{\tilde r/2}} \big\|_{L_y^{r/2}}^{r/2} \\
\nonumber&\lesssim\int_{|y|\le 1+\sqrt{t}}   |t|^{\frac{r}{2} (\frac{n-1}{\tilde r}-\frac n2+\sigma)}    dy\\
 &\qquad+ \int_{|y| \ge 1+\sqrt{t}} (|y|-1)^{r(-\sigma+\frac{n-1}{\tilde r})}|t|^{-\frac{r}{2} (\frac n2-2\sigma)}  dy.
\end{align}
The first integral in the right-hand side of \eqref{qwd} is bounded as
\begin{align}\label{g1}
\nonumber
\int_{|y|\le 1+\sqrt{t}}   |t|^{\frac{r}{2} (\frac{n-1}{\tilde r}-\frac n2+\sigma)}    dy
&=|t|^{\frac{r}{2} (\frac{n-1}{\tilde r}-\frac n2+\sigma)} \int_0^{1+\sqrt{t}} \rho^{n-1} d\rho \\
&\lesssim |t|^{\frac{r}{2} (\frac{n-1}{\tilde r}-\frac n2+\sigma)}  (1+ \sqrt{t})^n.
\end{align}
For the second inequality, we use the binomial theorem with the binomial coefficients $C_{n,k}$ to obtain
\begin{align}
\nonumber
\int_{|y| \ge 1+\sqrt{t}} (|y|-1)^{r(-\sigma+\frac{n-1}{\tilde r})}  dy
&=\int_{\sqrt{t}}^\infty {\rho}^{r(-\sigma+\frac{n-1}{\tilde r})} (\rho+1)^{n-1} d\rho \\
\nonumber
&=\sum_{k=0}^{n-1} C_{n,k}\int_{\sqrt{t}}^\infty {\rho}^{r(-\sigma+\frac{n-1}{\tilde r})+k} d\rho \\
\nonumber
&\lesssim {\sqrt{t}}^{r(-\sigma+\frac{n-1}{\tilde r})+1} \sum_{k=0}^{n-1} C_{n,k} {\sqrt{t}}^k \\
\label{g2}
&={|t|}^{\frac{r}{2} (-\sigma+ \frac{n-1}{\tilde r})+\frac{1}{2}} (1+\sqrt{t})^{n-1}.
\end{align}
Here, for the third inequality, we used the fact that $r(-\sigma+\frac{n-1}{\tilde r})+k+1<0$ for all $0\leq k\leq n-1$.
Indeed, this fact follows from the condition \eqref{c3}.

Combining \eqref{qwd}, \eqref{g1} and \eqref{g2}, we have
\begin{equation*}
\|K_t\|_{W(L^{\frac{\tilde r}{2}},L^{\frac{r}{2}})_x}\lesssim
|t|^{-\frac{n}{2}+\sigma+\frac{n-1}{\tilde r}} \big((1+\sqrt{t})^n +|t|^{\frac{1}{2}}(1+\sqrt{t})^{n-1} \big)^{\frac{2}{r}}.
\end{equation*}
Hence we get \eqref{x} as desired.

\subsection{ The case $n/4 \le \sigma < n/2$}
The estimate \eqref{x} in this case is proved in the same way as in the previous case.
From \eqref{ker} and \eqref{kerbd} with $\gamma=2\sigma$, we see
\begin{equation}\label{upbd2}
|K_t (x)| \lesssim
\begin{cases}
|t|^{-{\frac{n}{2}} +\sigma} \quad\text{if} \quad|x| \le \sqrt{t},\\
|x|^{-n+2\sigma} \quad\text{if} \quad|x| \ge\sqrt{t}.
\end{cases}
\end{equation}
To calculate $\|K_t\|_{W(L^{\frac{\tilde r}{2}}, L^{\frac{r}{2}})_x}$ using \eqref{upbd2},
we choose $\varphi(x) \in C_0^{\infty}(\mathbb{R}^n)$ supported on $\{x\in\mathbb{R}^n:|x|\le1\}$
and divide $\|K_t \tau_y \varphi\|_{L_x^{\tilde r/2}}$ into two cases, $|y|\le 1+\sqrt{t}$ and $|y|\ge 1+\sqrt{t}$, as before.

First we consider the case $|y|\le 1+\sqrt{t}$.
By using \eqref{upbd2} and the support condition of $\varphi$,

\begin{align}
\nonumber
\|K_t \tau_y \varphi\|_{L_x^{\tilde r/2}}^{\tilde r/2}
&\lesssim \int_{|y|-1 \le |x| \le \sqrt{t}} |t|^{\frac{\tilde r}{2} ({- \frac{n}{2}+\sigma})} dx + \int_{\sqrt{t} \le |x| \le |y|+1} |x|^{\frac{\tilde r}{2} (-n+2\sigma)} dx \\
\nonumber
&= |t|^{\frac{\tilde r}{2} (-\frac{n}{2} + \sigma)} \int_{|y|-1}^{\sqrt t} \rho^{n-1} d\rho + \int_{\sqrt t}^{|y|+1} \rho^{{\tilde r} (-\frac{n}{2}+\sigma+\frac{n-1}{\tilde r})} d\rho \\
\nonumber
&= |t|^{\frac{\tilde r}{2} (-\frac{n}{2} + \sigma)} \frac{{\sqrt{t}}^n - (|y|-1)^n }{n}\\
\nonumber
& \qquad + \frac{ (|y|+1)^{{\tilde r} (-\frac{n}{2}+\sigma+\frac{n-1}{\tilde r})+1}
- \sqrt{t}^{{\tilde r} (-\frac{n}{2}+\sigma+\frac{n-1}{\tilde r})+1} }{{\tilde r} (-\frac{n}{2}+\sigma+\frac{n-1}{\tilde r})+1}.
\end{align}
Since $n-1\ge 0$ and $\tilde r(-\frac{n}{2}+\sigma+\frac{n-1}{\tilde r})<0$ from \eqref{c4},
by applying the mean value theorem as before, we now see

\begin{align}\label{I'}
\nonumber
\|K_t \tau_y \varphi\|_{L_x^{\tilde r/2}}^{\tilde r/2}
&\lesssim |t|^{\frac{\tilde r}{2} (-\frac{n}{2} + \sigma +\frac{n-1}{\tilde r})} (\sqrt{t}-|y|+1)
+ |t|^{\frac{\tilde r}{2}(-\frac{n}{2}+\sigma+\frac{n-1}{\tilde r})} (|y|+1-\sqrt{t}) \\
&= 2|t|^{\frac{\tilde r}{2}(-\frac{n}{2}+\sigma+\frac{n-1}{\tilde r})}
\end{align}
when $|y|\le 1+\sqrt{t}$.
Similarly, for the other case $|y|\ge 1+\sqrt{t}$, we get
\begin{align}\label{II'}
\nonumber
\|K_t \tau_y \varphi\|_{L_x^{\tilde r/2}}^{\tilde r/2} &\lesssim \int_{|y|-1 \le |x| \le |y|+1} |x|^{\frac{\tilde r}{2} (-n+2\sigma)} dx \\
\nonumber
&= \int_{|y|-1}^{|y|+1} \rho^{\tilde r(- \frac{n}{2}+\sigma+\frac{n-1}{\tilde r})} d\rho \\
\nonumber
&= \frac{ (|y|+1)^{\tilde r(- \frac{n}{2}+\sigma+\frac{n-1}{\tilde r})+1}-(|y|-1)^{\tilde r(- \frac{n}{2}+\sigma+\frac{n-1}{\tilde r})+1}}
{\tilde r(- \frac{n}{2}+\sigma+\frac{n-1}{\tilde r})+1} \\
&\lesssim (|y|-1)^{\tilde r(- \frac{n}{2}+\sigma+\frac{n-1}{\tilde r})}.
\end{align}

By \eqref{I'} and \eqref{II'}, it follows now that
\begin{align}\label{qwd2}
\nonumber\|K_t\|_{W(L^{\frac{\tilde r}{2}} ,L^{\frac{r}{2}})_x}^{r/2}
&= \big\|\, \|K_t \tau_y \varphi\|_{L_x^{\tilde r/2}} \big\|_{L_y^{r/2}}^{r/2} \\
\nonumber&\lesssim\int_{|y|\le 1+\sqrt{t}}   |t|^{\frac{r}{2}(-\frac{n}{2}+\sigma+\frac{n-1}{\tilde r})}    dy\\
 &\qquad+ \int_{|y| \ge 1+\sqrt{t}} (|y|-1)^{r(-\frac{n}{2}+\sigma+\frac{n-1}{\tilde r})}  dy.
\end{align}
The first integral in the right-hand side of \eqref{qwd2} is bounded as
\begin{align}\label{g3}
\nonumber
\int_{|y|\le 1+\sqrt{t}} |t|^{\frac{r}{2}(-\frac{n}{2}+\sigma+\frac{n-1}{\tilde r})} dy
&=|t|^{\frac{r}{2}(-\frac{n}{2}+\sigma+\frac{n-1}{\tilde r})} \int_0^{1+\sqrt{t}} \rho^{n-1} d\rho \\
&\lesssim |t|^{\frac{r}{2}(-\frac{n}{2}+\sigma+\frac{n-1}{\tilde r})} (1+ \sqrt{t})^n.
\end{align}
For the second inequality, we use the binomial theorem with the binomial coefficients $C_{n,k}$ to obtain
\begin{align} \label{g4}
\nonumber\int_{|y|\ge 1+\sqrt{t}} (|y|-1)^{r(-\frac{n}{2}+\sigma+\frac{n-1}{\tilde r})} dy
&=\int_{\sqrt{t}}^\infty {\rho}^{r(-\frac{n}{2}+\sigma+\frac{n-1}{\tilde r})} (\rho+1)^{n-1} d\rho \\
\nonumber
&= \sum_{k=0}^{n-1} C_{n,k} \int_{\sqrt{t}}^\infty {\rho}^{r(-\frac{n}{2}+\sigma+\frac{n-1}{\tilde r})+k} d\rho \\
\nonumber
&\lesssim {\sqrt{t}}^{r(-\frac{n}{2}+\sigma+\frac{n-1}{\tilde r})+1} \sum_{k=0}^{n-1} C_{n,k} {\sqrt{t}}^k \\
&=|t|^{\frac{r}{2}(-\frac{n}{2}+\sigma+\frac{n-1}{\tilde r})+\frac{1}{2}} (1+\sqrt{t})^{n-1}.
\end{align}
Here, for the third inequality, we used the fact that $r(-\frac{n}{2}+\sigma+\frac{n-1}{\tilde r})+k+1<0$ for all $0\leq k\leq n-1$.
Indeed, this fact follows from the condition \eqref{c4}.

Combining \eqref{qwd2}, \eqref{g3} and \eqref{g4}, we have
\begin{equation*}
\|K_t\|_{W(L^{\frac{\tilde r}{2}},L^{\frac{r}{2}})_x}\lesssim
|t|^{-\frac{n}{2}+\sigma+\frac{n-1}{\tilde r}} \big((1+\sqrt{t})^n +|t|^{\frac{1}{2}}(1+\sqrt{t})^{n-1} \big)^{\frac{2}{r}}.
\end{equation*}
Hence we get \eqref{x} as desired.

\section{Concluding remarks} \label{sec4}
In this final section, we discuss Remark \ref{rmk} in detail.
As mentioned there, we can obtain further estimates by complex interpolation (see \eqref{inter})
between bilinear form estimates given from \eqref{C-N2} and \eqref{T}.
Here we explain this only for the particular case where we use \eqref{C-N} and
\eqref{T} with $\tilde r=\infty$ instead of \eqref{C-N2} and \eqref{T}, respectively.
This is strictly intended to make the argument shorter, and one could adapt the same argument from this case to handle the other cases as well.

\begin{cor} \label{cor}
Let $n\geq1$ and $\max\{0,(n-2)/8\}<\sigma<n/4$.
Assume that $(q,r)$ satisfy \eqref{cc1},
\begin{equation}\label{c2c2}
\frac2{\widetilde{q}}>\frac n4-\sigma,\quad 0<\frac1q<\frac1{\widetilde{q}}+\frac14\leq\frac12\quad\text{and}\quad 2\leq r\leq\infty.
\end{equation}
Here, $r\neq\infty$ if $n=2$.
Then we have
\begin{equation} \label{fur}
\| e^{it\Delta} f\|_{W(L^{\tilde q} ,L^q)_t W(L^4 ,L^r)_x} \lesssim \|f\|_{\dot H^{\sigma}}.
\end{equation}
\end{cor}

\begin{rem}
When $\tilde r = 4$, the possible range of $\tilde q$ in Theorem \ref{thm} is
$2 \leq \tilde q < 8/(n-4\sigma+1)$.
On the other hand, the possible range of $\tilde q$ in the above corollary is
$4\leq\widetilde{q}<8/(n-4\sigma)$ if $\sigma>\max\{0,(n-2)/4\}$.
Since $8/(n-4\sigma+1)< 8/(n-4\sigma)$, it gives further estimates which do not follow from Theorem \ref{thm}.
For fixed $q$, the exponent $r$ given from \eqref{cc1} is smaller than $r$ given from \eqref{c2} with $\tilde r =4$.
From this observation and the inclusion relation \eqref{inclusion},
we also note that \eqref{fur} is stronger than the estimate \eqref{T} with $\tilde r =4$ in Theorem \ref{thm}.
Similarly for fixed $r$.
\end{rem}

\begin{proof}[Proof of Corollary \ref{cor}]
Firstly, we recall from Subsection \ref{subsec2.1} that the standard $TT^\ast$ argument gives that
\begin{equation*}
\| e^{it\Delta} f\|_{W(L^{\tilde q} ,L^q)_t W(L^{\tilde r} ,L^r)_x} \lesssim \|f\|_{\dot H^{\sigma}}
\end{equation*}
is equivalent to the estimate \eqref{TT*} which is in turn equivalent to the following bilinear form estimate
\begin{equation} \label{bi}
|T(F,G)| \lesssim \|F\|_{W(L^{\tilde q'},L^{q'})_t W(L^{\tilde r'}, L^{r'})_x} \|G\|_{W(L^{\tilde q'},L^{q'})_t W(L^{\tilde r'}, L^{r'})_x},
\end{equation}
where
\begin{equation*}
T(F,G) :=\int_{\mathbb{R}}\int_{\mathbb{R}} \big\langle e^{-is\Delta}|\nabla|^{-\sigma}F(\cdot,s), e^{-it\Delta}|\nabla|^{-\sigma}G(\cdot,t)\big\rangle_x dsdt.
\end{equation*}

Next, using the Cauchy-Schwarz inequality,
\begin{align}\label{CS}
\nonumber|T(F,G)|&=\bigg| \bigg< \int_{\mathbb{R}} e^{-is\Delta}|\nabla|^{-2\sigma}F(\cdot,s) ds, \int_{\mathbb{R}} e^{-it\Delta}G(\cdot,t) dt \bigg>_x \bigg| \\
&\le \bigg\| \int_{\mathbb{R}} e^{-is\Delta}|\nabla|^{-2\sigma}F(\cdot,s) ds \bigg\|_{L_x^2} \bigg\|\int_{\mathbb{R}} e^{-it\Delta}G(\cdot,t) dt \bigg\|_{L_x^2}.
\end{align}
For the first $L_x^2$ norm in \eqref{CS}, we use the estimate \eqref{T*}
with $\tilde r=\infty$ and $\sigma$ replaced by $2\sigma$,
\begin{equation} \label{bfint1}
\bigg\| \int_{\mathbb{R}} e^{-is\Delta} |\nabla|^{-2\sigma} F(\cdot,s) ds \bigg\|_{L^2_x}
\lesssim \| F\|_{W(L^{\tilde {q}_1'} , L^{q_1'})_t W(L^1 , L^{r_1'})_x},
\end{equation}
where $2 \leq \tilde {q}_1 < q_1 < \infty$, $2\leq r_1 \leq \infty$, $\max\{0,(n-2)/8\} < \sigma < n/4$,
\begin{equation}\label{ccc}
\frac{2}{\tilde {q}_1} > \frac{n}{2} -2\sigma \quad\text{and}\quad
\frac{2}{q_1} + \frac{n}{r_1} = \frac{n}{2}-2\sigma.
\end{equation}
For the second $L_x^2$ norm in \eqref{CS}, we use the following dual estimate of \eqref{C-N},
\begin{equation} \label{bfint2}
\bigg\|\int_{\mathbb{R}} e^{-it\Delta}G(\cdot,t) dt \bigg\|_{L^2_x} \lesssim \|G\|_{W(L^1, L^{q_2'})_t W(L^2, L^{r_2'})_x},
\end{equation}
where $(q_2 ,r_2)$ is Schr\"odinger admissible (see \eqref{admiss}).
Combining \eqref{CS}, \eqref{bfint1} and \eqref{bfint2}, we then have
\begin{equation} \label{inter1}
|T(F,G)| \lesssim \|F\|_{W(L^{\widetilde{q}_1'}, L^{q_1'})_t W(L^1, L^{r_1'})_x} \|G\|_{W(L^1, L^{q_2'})_t W(L^2, L^{r_2'})_x},
\end{equation}
and by symmetry
\begin{equation} \label{inter2}
|T(F,G)|\lesssim \|F\|_{W(L^1, L^{q_2'})_t W(L^2, L^{r_2'})_x} \|G\|_{W(L^{\tilde{q}_1'}, L^{q_1'})_t W(L^1, L^{r_1'})_x},
\end{equation}
for $\widetilde{q}_1$, $(q_1,r_1)$ and $(q_2,r_2)$ given as above.
Finally, by applying the complex interpolation \eqref{inter} with $\theta=1/2$ between \eqref{inter1} and \eqref{inter2},
we obtain \eqref{bi} for
\begin{equation*}
\frac{1}{\tilde q} = \frac{1}{2\tilde{q}_1}, \quad \frac{1}{q} = \frac{1}{2} \Big(\frac{1}{q_1} + \frac{1}{q_2} \Big), \quad \frac{1}{\tilde r}=\frac{1}{4}, \quad \frac{1}{r}=\frac{1}{2} \Big(\frac{1}{r_1} + \frac{1}{r_2} \Big).
\end{equation*}
Combining the second condition in \eqref{ccc} and $2/{q_2} + n/{r_2} = n/2$ implies the condition \eqref{cc1}.
From the first condition in \eqref{ccc}, we see the first condition in \eqref{c2c2}.
Since $\widetilde{q}_1<q_1<\infty$ and $2\leq q_2\leq\infty$, $0<1/q<1/4+1/(2\widetilde{q}_1)=1/4+1/\widetilde{q}$.
Since $\widetilde{q}_1\geq2$, it follows also that $\widetilde{q}\geq4$. Hence we see the second condition in \eqref{c2c2}.
From the conditions $2\leq r_1\leq\infty$ and $2\leq r_2\leq\infty$, we finally see $2\leq r\leq\infty$.
Here, $r\neq\infty$ if $n=2$ since $r_2\neq\infty$ if $n=2$.
This determines the last condition in \eqref{c2c2}.
\end{proof}



\begin{thebibliography}{99}

\bibitem{BBCRV} J. A. Barcel\'o, J. M. Bennett, A. Carbery, A. Ruiz and M. C. Vilela, \textit{Strichartz inequalities with weights in Morrey-Campanato classes}, Collect. Math. 61 (2010), 49-56.

\bibitem{CN} E. Cordero and F. Nicola, \textit{Strichartz estimates in Wiener amalgam spaces for the Schr\"odinger equation},
Math. Nachr. 281 (2008), 25-41.

\bibitem{CN2} E. Cordero and F. Nicola, \textit{Metaplectic representation on Wiener amalgam spaces
and applications to the Schr\"odinger equation}, J. Funct. Anal. 254 (2008), 506-534.

\bibitem{CN3} E. Cordero and F. Nicola, \textit{Some new Strichartz estimates for the Schr\"odinger equation},
J. Differential Equations 245 (2008), 1945-1974.



\bibitem{F} H. G. Feichtinger, \textit{Banach convolution algebras of Wiener type},
Functions, series, operators, Vol I, II (Budapest, 1980), 509-524,
Colloq. Math. Soc. J\'anos Bolyai, 35, North-Holland, Amsterdam, 1983.

\bibitem{F2} H. G. Feichtinger, \textit{Banach spaces of distributions of Wiener's type and interpolation},
Functional analysis and approximation (Oberwolfach, 1980), pp. 153-165,
Internat. Ser. Numer. Math., 60, Birkh\"auser, Basel-Boston, Mass., 1981.

\bibitem{F3} H. G. Feichtinger, \textit{Generalized amalgams, with applications to Fourier transform},
 Canad. J. Math. 42 (1990), 395-409.

\bibitem{GV} J. Ginibre and G. Velo, \textit{The global Cauchy problem for the nonlinear Schr\"odinger equation
revisited}, Ann. Inst. H. Poincar\'{e} Anal. Non Lin\'{e}aire 2 (1985), 309-327.

\bibitem{H} C. Heil, \textit{An introduction to weighted Wiener amalgams}, Wavelets and Their Applications (M. Krishna, R. Radha and S. Thangavelu, eds.), Allied Publishers Private Limited, (2003), pp.183-216.

\bibitem{KT} M. Keel and T. Tao, \textit{Endpoint Strichartz estimates}, Amer. J. Math. 120 (1998), 955-980.

\bibitem{M-S} S. J. Montgomery-Smith, \textit{Time decay for the bounded mean oscillation of solutions of the Schr\"odinger and wave equations}, Duke Math. J. 91 (1998), 393-408.

\bibitem{S} I. Seo, \textit{Unique continuation for the Schr\"odinger equation with potentials in Wiener amalgam spaces},
Indiana Univ. Math. J. 60 (2011), 1203-1227.

\bibitem{St} E. M. Stein, \textit{Singular Integrals and Differentiability Properties of Functions}, Princeton Univ. Press. Princeton, (1970).

\bibitem{Str} R. S. Strichartz, \textit{Restrictions of Fourier transforms to quadratic surfaces and decay of solutions of wave equations}, Duke Math. J. 44 (1977), 705-714.

\bibitem{T} T. Tao, \textit{Low regularity semi-linear wave equations}, Comm. Partial Differential Equations 24 (1999), 599-629.


\end{thebibliography}
\end{document}